\documentclass[letterpaper, 10pt, conference]{ieeeconf}  

\IEEEoverridecommandlockouts                              
\usepackage{amssymb,latexsym,epsfig,epic,eepic}
\usepackage{amsmath,amsfonts,psfrag}
\usepackage{dsfont} 
\usepackage{balance}

\newtheorem{theorem}{Theorem}

\newtheorem{lemma}[theorem]{Lemma}
\newtheorem{proposition}[theorem]{Proposition}

\newtheorem{remark}{Remark}

\newtheorem{example}{Example}

\newtheorem{assumption}{Assumption}

\newcommand{\reeq}[1]{(\ref{eq.#1})}
\newcommand{\ha}{\frac{1}{2}}
\newcommand{\diag}  {\ensuremath {\mathrm{diag}}}
\newcommand{\ip}[2]{\langle{#1},{#2}\rangle}
\newcommand{\Lt}{  \mathcal{L}_2 }
\newcommand{\Ht}{ \mathcal{H}_{2}}
\newcommand{\Hinf}{ \mathcal{H}_{\infty}}

\newcommand{\R}{{\mathbb R}}

\newcommand{\tr}{\mbox{\rm Tr}}

\newcommand{\btmz}{\begin{itemize}}
\newcommand{\etmz}{\end{itemize}}
\newcommand{\benum}{\begin{enumerate}}
\newcommand{\eenum}{\end{enumerate}}

\newcommand{\bmat}[1]{\begin{bmatrix}#1\end{bmatrix}}

\newcommand{\stsp}[4]{
\left[ \begin{array}{c|c}
       #1 & #2 \\ \hline
       #3 & #4
       \end{array} \right]}

\usepackage{ifthen}
\usepackage[usenames,dvipsnames]{color}
\newboolean{showcomments}
\setboolean{showcomments}{true}
\newcommand{\enrique}[1]{\ifthenelse{\boolean{showcomments}}
{\textcolor{Red}{(Enrique says: #1)}}{}}
\newcommand{\addcite}[0]{\ifthenelse{\boolean{showcomments}}
{\textcolor{Purple}{(add cite(s)) }}{}}%
\newcommand{\highlight}[1]{\ifthenelse{\boolean{showcomments}}
{\textcolor{blue}{#1}}{#1}}

\allowdisplaybreaks

\renewcommand{\baselinestretch}{.985}

\title{\LARGE \bf Global performance metrics for synchronization of heterogeneously rated power systems:
The role of machine models and inertia}

\author{Fernando Paganini,~\IEEEmembership{Fellow,~IEEE}, and Enrique Mallada,~\IEEEmembership{Member,~IEEE}
\thanks{Fernando Paganini is with Universidad ORT Uruguay, Montevideo. e-mail: paganini@ort.edu.uy. Enrique Mallada is with Johns Hopkins University, Baltimore. e-mail: mallada@jhu.edu. This work was supported in part by IDB/MIEM-Uruguay, Project ATN/KF 13883 UR,  ANII-Uruguay, grant FSE\_1\_2016\_1\_131605,
and NSF, grants CNS 1544771, EPCN 1711188,  and AMPS 1736448.
}
}

\begin{document}

\maketitle

\begin{abstract}
A recent trend in control of power systems has sought to quantify the synchronization dynamics in terms of a global performance metric, compute it under very simplified assumptions, and use it to gain insight on the role of system parameters, in particular, inertia. In this paper, we wish to extend this approach to more realistic scenarios, by incorporating the heterogeneity of machine ratings, more complete machine models, and also to more closely map it to classical power engineering notions such as Nadir, Rate of Change of Frequency (RoCoF), and inter-area oscillations.

We consider the system response to a step change in power excitation, and define the system frequency as a weighted average of generator frequencies (with weights proportional to each machine's rating); we characterize Nadir and RoCoF by the $L_\infty$ norm of the system frequency and its derivative, respectively, and inter-areas oscillations by the $L_2$ norm of the error of the vector of bus frequencies w.r.t. the system frequency.

For machine models where the dynamic parameters (inertia, damping, etc.) are proportional to rating, we analytically compute these norms and use them to show that the role of inertia is more nuanced than in the conventional wisdom. With the classical swing dynamics, inertia constant plays a secondary role in performance. It is only when the turbine dynamics are introduced that the benefits of inertia become more prominent.

\end{abstract}

\section{Introduction}\label{sec.intro}

The synchronization performance of the power grid has been a major
concern of system operators since the early days~\cite{Evans:1924hr,Steinmetz:1920fg}.  Most
generators and loads are designed on the assumption that the grid
frequency is tightly regulated around a nominal value (e.g. 60Hz
in the U.S., 50Hz in Europe). When the frequency deviates
significantly (a few hundred mHz) due to some network fault,
several mechanisms, such as  machine protections or under frequency load shedding
(ULFS)~\cite{NERC:2015tc}, automatically
disconnect critical network elements, potentially causing
cascading failures and ultimately blackouts~\cite{FederalEnergyRegulatoryCommission:2012wq}.

The gradual substitution of conventional electromechanical
 with renewable generation has raised new concerns about
synchronization performance. The former provide a natural response
to power imbalances, which is not present in the inverter-based
interfaces of renewable sources; in particular the lack of
\emph{inertia} in the latter is often seen as a threat to
frequency regulation. An objective analysis of this issue requires
identifying appropriate performance \emph{metrics}.

Two separate mechanisms are at play in frequency fluctuations. On
one hand, abrupt changes on the global supply-demand balance
induce system-wide frequency changes, which may persist in steady
state. On the other hand, geographically-distributed frequency
oscillations (a.k.a. inter-area oscillations)~\cite{Hsu:di,
klein_fundamental_1991, messina_inter-area_2011} may be observed
due to weak global coupling. Good performance metrics should be
able to identify and discriminate between these two phenomena.

 To address these problems, power engineers have traditionally
relied on classical control metrics. For example, they use the
system response to a step input to measure the maximum frequency
deviation to imbalances (Nadir) as well as the maximum
rate of change of frequency (RoCoF)\cite{Miller:2011tm}; their main
limitation is that these quantities are node dependent. To
evaluate inter-area oscillations, eigenvalue methods (slow
coherency \cite{Winkelman:1981}, participation factors \cite{Verghese:1982}) have
been employed.

More recently, a trend in control of power systems has aimed to
quantify synchronization performance in terms of \emph{global}
system metrics such as $\mathcal{H}_2$ or $\mathcal{H}_\infty$
\cite{bassam,mevsanovic2016comparison,m2016cdc,poolla_dorfler2017,simpson-porco2017,andreasson2017,jpm2017cdc}.
These norms capture the effect of system
parameters, such as inertia, damping, and the eigenvalues of the
network (Laplacian) matrix, on system performance. However,
closed-form analytic results depend on oversimplified assumptions
--homogeneous machines modeled by swing equations-- and do not
directly represent step-response information which is most
important for network operators to analyze disturbances.

In this paper, we wish to bridge the gap between the two
approaches. Firstly, to extend the latter approach to cover
more realistic scenarios  by incorporating heterogeneity of
machine ratings, and adding turbine dynamics. Secondly, to
incorporate step-response metrics (Nadir, RoCoF) for an
appropriately defined global system frequency, and to separately
characterize inter-area oscillations.

To obtain analytically tractable results, we focus on a
specifically family of heterogeneous machines, in which key
dynamic parameters are proportional to rating. This restriction is
mild in comparison with homogeneity, and enables a diagonalization
procedure, generalizing traditional eigen-analysis. From it, a
\emph{system frequency} suitable for step response analysis
appears naturally, and turns out to be
\begin{align}\label{eq.coi}
\bar w(t) :=  \frac{\sum_i m_i w_i(t)}{\sum_i m_i},
\end{align}
the weighted average of  node frequencies in proportion to their
inertia. This is identical to the frequency of the center of  inertia
(COI) a classical notion \cite{kundur_power_1994,bergen_vittal_2000}. Nadir and RoCoF are defined as
the $\mathcal{L}_\infty$ norms of, respectively, $\bar{w}(t)$ and
$\dot{\bar w}(t)$. A synchronization cost measuring transient
inter-area oscillations is defined as the $\Lt$ norm of the vector
of deviations $\tilde w_i(t) = w_i(t) -\bar w(t)$.

The rest of the paper is organized as follows. In Section
\ref{sec.diag} we formulate the model and carry out the
diagonalization of the dynamics, making the above decomposition
precise. The system frequency step response, and a closed form
expression for the synchronization cost, are both expressed in
terms of a representative machine and the network structure.

In Sections \ref{sec.swing} and \ref{sec.turbine} we apply the
results to different machine models, respectively to the
second-order swing dynamics, and a third-order model that
incorporates the turbine control. We find that some important
aspects, in particular the importance of inertia, depend crucially
on the chosen model. The paper concludes with a discussion in
Section \ref{sec.concl}, and some derivations are covered in the
Appendix.

\section{Dynamic model with machine heterogeneity}\label{sec.diag}

\begin{figure}[htb]
\centering
\includegraphics[width=.75\columnwidth]{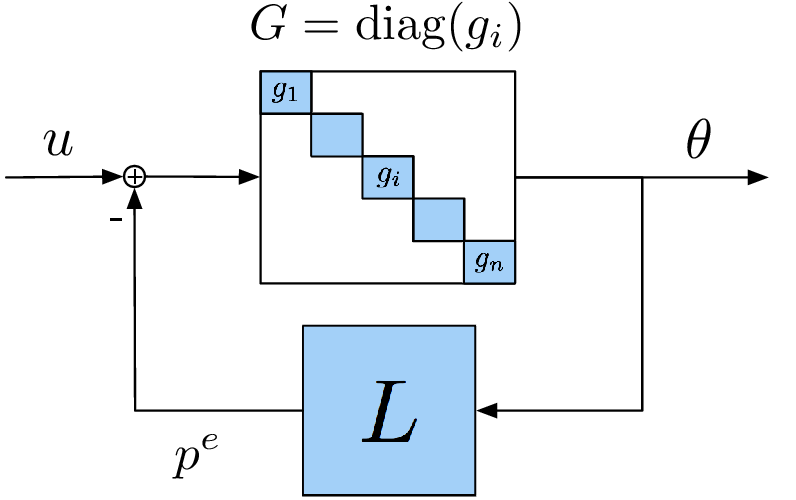}
\caption{Block Diagram of Linearized Power Network}\label{fig.GL}
\end{figure}

We consider a set of $n$ generator buses, indexed by $i\in\{1,\dots,n\}$, dynamically coupled through an AC network.
Assuming operation around an equilibrium, the linearized dynamics are represented by the
block diagram in Fig. \ref{fig.GL}, where:
\btmz
\item $G(s) = \diag(g_i(s))$ is the diagonal transfer function of generators at each bus.
Each $g_i(s)$ has as output the phase angle $\theta_i$, and as
input the net power at its generator axis, relative to its
equilibrium value. This includes an outside disturbance $u_i$,
reflecting variations in mechanical power or local load, minus the
variation $p^e_i$ in electrical power drawn from the network.
\item Using a linear DC model for the network, the vector of drawn power is written as $p^e = L \theta$, where
$L$ is the weighted Laplacian of the graph defined by the line
susceptances. Thus all the coupling between the bus subsystems is through this
feedback term. $L$ is a rank $n-1$ matrix with kernel spanned by
$\mathbf{1}$, the vector of all ones.  \etmz
Two examples of generator dynamics, to be
considered explicitly in this paper, are:
\begin{example}\label{ex.swing}
The swing equation dynamics
\begin{align*}
\dot{\theta_i} & = w_i, \\
m_i \dot{w_i} &= - d_i w_i + u_i - p^e_i.
\end{align*}
This corresponds to the transfer function
\begin{align}\label{eq.giswing}
g_i(s) = \frac{1}{m_i s^2 + d_i s}.
\end{align}
\end{example}
\mbox{}

\begin{example}\label{ex.turbine}
The swing equation with a first-order model of the turbine control:
\begin{align*}
\dot\theta_i &= w_i,\\
m_i\dot w_i &= - d_i w_i + q_i + u_i - p^e_i,\\
\tau_i \dot q_i &= -r^{-1}_i w_i -q_i.
\end{align*}
Here $q_i$ is the (variation of) turbine power, $\tau_i$ the turbine time constant and $r_i$ the droop coefficient; the governor
dynamics are considered to be faster and neglected. The corresponding transfer function is
\begin{align}\label{eq.giturbine}
g_i(s) = \frac{\tau_i s + 1}{s \left(m_i \tau_i s^2 + (m_i+d_i
\tau_i) s+d_i +r_i^{-1}\right)}.
\end{align}
\end{example}
\mbox{}

Of course, other models are possible within this framework (e.g. a
4th order system including a state for the governors).

\subsection{A family of heterogeneous machines}

A popular research topic in recent years
\cite{bassam,mevsanovic2016comparison,m2016cdc,poolla_dorfler2017,simpson-porco2017,andreasson2017,jpm2017cdc} has
been the application of global metrics from robust control to this
kind of synchronization dynamics, as a tool to shed light on the
role of various parameters, e.g. system inertia. Most of the analytical
results, however, consider a \emph{homogeneous} network where all
machines are identical (i.e., common $m_i$, $d_i$, etc.), a very
restrictive scenario.\footnote{Some bounds on heterogenous systems
are given in \cite{bassam,poolla_dorfler2017}. Numerical studies with heterogeneity are given in \cite{mevsanovic2016comparison}.}

In a real network, where generators have different power
{ratings}, it is natural  for parameters to scale accordingly:
for instance, the inertia $m_i$ of a machine will
grow with its rating, and it is clear that ``heavier" machines
will have a more significant impact in the overall dynamics.

While in principle one would like to cover general parameters, we
show here that a compact analysis can be given for the case where
parameters satisfy a certain proportionality. We formalize this
by introducing a  \emph{rating parameter} $0< f_i\leq 1$, defined in
relation to the largest machine which has $f_i=1$, and imposing the
following:
\begin{assumption}\label{ass.scale}
There exists a fixed transfer function $g_0(s)$, termed the \emph{representative machine}, such that
\[
g_i(s) = \frac{1}{f_i} g_0(s)
\]
for each $i$, where $f_i > 0$ is the rating parameter of bus $i$.
\end{assumption}

To interpret this, consider first the swing dynamics of Example
\ref{ex.swing}. Then the assumption is satisfied provided
that inertia and damping are both proportional\footnote{This kind of proportionality is termed ``uniform damping" in \cite{bergen_vittal_2000}.} to $f_i$, i.e. $m_i = f_i  m$, $d_i = f_i d$,
where $m$, $d$ are those of the largest (representative) machine,
\[
g_0(s) = \frac{1}{ms^2 +ds}. \]

Going to the case of Example \ref{ex.turbine} with the turbine
dynamics, we find that the above assumption is satisfied provided
that $m_i = f_i  m$, $d_i = f_i d$, $r_i^{-1} = f_i  r^{-1}$,
$\tau_i = \tau$; here the inverse droop coefficient is assumed
proportional to rating, but the turbine time-constant is taken to
be homogeneous.

The corresponding representative machine is
\[
g_0(s) = \frac{\tau s + 1}{s \left(m \tau s^2 + (m+d \tau) s+d
+r^{-1}\right)}. \]

Regarding the practical relevance of our simplifying assumption:
empirical values reported in \cite{oakridge2013} indicate that at least in
regard to orders of magnitude, proportionality is a reasonable
first-cut approximation to heterogeneity, substantially more
realistic than the homogeneous counterpart.

\subsection{Diagonalization}

We will now exploit the above assumption to transform the dynamics
of Fig. \ref{fig.GL} in a manner that allows for a suitable
decoupling in the analysis. In what follows, $F = \diag(f_i)$
denotes the diagonal matrix of rating parameters. Writing
\[
G(s) = \diag(g_i(s)) = F^{-\ha} [g_0(s) I] F^{-\ha},
\]
we transform the feedback loop into the equivalent form of Fig. \ref{fig.LF}.

\begin{figure}[htb]
\centering
\includegraphics[width=1\columnwidth]{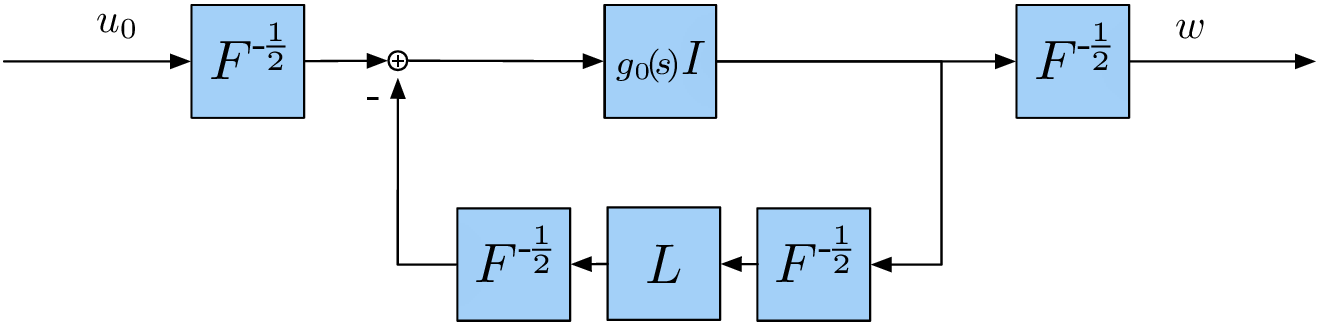}
\caption{Equivalent block diagram for heterogeneously rated machines}\label{fig.LF}
\end{figure}

We introduce a notation for the scaled Laplacian matrix\footnote{This scaling already appears in the classical paper \cite{Winkelman:1981}.}
\begin{align}\label{eq.lf}
L_F:=  F^{-\ha} L F^{-\ha},
\end{align}
which is positive semidefinite and of rank $n-1$. Applying the spectral theorem we diagonalize it as
\begin{align}\label{eq.lfdiag}
L_F= V \Lambda V^T,
\end{align}
where $\Lambda = \diag(\lambda_k), \quad 0 = \lambda_0 <
\lambda_1\leq \cdots \leq \lambda_{n-1},$  and $V$ is unitary.
Distinguishing the eigenvector $v_0$ that corresponds to the zero
eigenvalue, we write $V=\left[v_0 \ \  V_\perp\right],$ where
\[
V_\perp \in \R^{n\times (n-1)}, \ V_\perp^T V_\perp = I_{n-1}, \
V_\perp^T v_0 = 0.\]
In fact $v_0$ can be made explicit by
recalling that $\ker(L) = \mathrm{span}\{\mathbf{1}\}$, so $\ker(L_F) =
\mathrm{span}\{F^{\ha} \mathbf{1}\}$, from where
\begin{align}\label{eq.v1}
v_0 = \alpha_F  F^\ha  \mathbf{1}, \quad \mbox{ with } \alpha_F:= \Big(\sum_{i} f_i \Big)^{-\ha}.
\end{align}

Substitution of \reeq{lfdiag} into Fig. \ref{fig.LF} and some
block manipulations leads to the equivalent representation of
Fig. \ref{fig.penult}.

\begin{figure}[hbt]
\centering
\includegraphics[width=1\columnwidth]{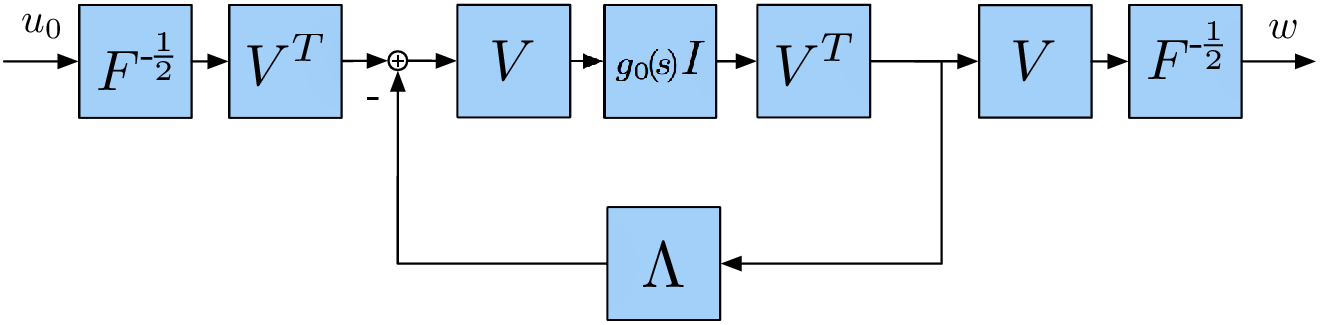}
\caption{Equivalent block diagram for heterogeneously rated machines with diagonalized closed loop}\label{fig.penult}
\end{figure}

Noting finally that the $V$ block commutes with $g_0(s)I$ and thus
cancels out with $V^T$, the internal loop is now fully
diagonalized, yielding the closed-loop transfer function

\begin{align}\label{eq.h}
H(s) &= \diag(h_k(s)), \quad \mbox{with} \nonumber\\
h_k(s) &= \frac{g_0(s)}{1 + \lambda_k g_0(s)}, \quad
 k=0,1,\ldots,n-1.
\end{align}

\begin{assumption}\label{ass.stable}
The proportional feedback $\lambda_k>0 , k=1, \ldots, n-1$ is
stabilizing for $g_0(s)$, i.e. $h_k(s)$ has all its poles in
$Re[s]<0$ for $k=1,\ldots, n-1$.
\end{assumption}

Later on we will verify that the assumption always holds for the
two examples considered.

 As for
$h_0(s)$, we see that here there is no feedback: $h_0(s) =
g_0(s)$. In the examples above, $g_0(s)$ has a pole at $s=0$,
which corresponds to the integration from frequency to angle. This
is the typical situation in which the \emph{local} control  of the
machine has no angle feedback, the latter only appears when
considering coupling through the network.

Returning to Fig. \ref{fig.penult} we arrive at the transfer function
\[
T_{\theta u}(s) = F^{-\ha} V  H(s) V^T  F^{-\ha}
\]
between the vector of external power disturbances and the machine angle outputs.
If the output of interest is chosen to be the vector of frequencies, the
relevant transfer function is
\[
T_{w u}(s) = s T_{\theta u}(s).
\]

\subsection{Step response characterization}

Global metrics for synchronization performance  in e.g. \cite{bassam,poolla_dorfler2017,simpson-porco2017},
 are system norms ($\Ht$, $\Hinf$) applied to $T_{wu}$ (frequency output), or to a
``phase coherency"  output based on differences in
output angles. The choice of metric carries an implicit assumption
on the power disturbances considered (white noise, or a worst-case
$\Lt$ signal).

In this paper, we wish to bring metrics closer to industry
practice, by considering a \emph{step} input disturbance (e.g. due
a fault), and analyzing the response of the vector of bus
frequencies $w_i(t)$. Practitioners are interested in some
time-domain performance metrics (Nadir, RoCoF, see below) for the
response, but transient frequency is bus-dependent.

A candidate global notion of \emph{system frequency} is the weighted
average $\bar{w}(t)$ in \reeq{coi}. We now show that for our family of heterogenous systems, the
behavior of $\bar{w}(t)$ decouples nicely from the individual bus
deviations $w_i(t)- \bar{w}(t)$, opening the door for a separate
analysis of both aspects.

Our input will be a step function $u(t) = u_0 \mathds {1}_{t\geq
0}$; here $u_0$ is a given vector direction. In Laplace
transforms:
\begin{align}\label{eq.stepresp}
w(s) = T_{w u}(s) \frac{1}{s} u_0 = F^{-\ha} V  H(s) V^T  F^{-\ha} u_0.
\end{align}
We now isolate the dynamics corresponding to the eigenvalue
$\lambda_0=0$ and its eigenvector $v_0$ from the rest:
\begin{align*}
w(s) =  F^{-\ha} v_0  h_0(s) v_0^T  F^{-\ha} u_0 +  F^{-\ha}
V_\perp  \tilde{H}(s) V_\perp^T  F^{-\ha} u_0;
\end{align*}
here $\tilde{H}(s) = \diag_{k=1,..,n-1}\left(h_k(s))\right)$.

Noting that $F^{-\ha} v_0 =\alpha_F  \mathbf{1}$ from \reeq{v1},
and $h_0(s)=g_0(s)$, the first term above is of the form
$\bar{w}(s)\mathbf{1}$, where
\begin{align}\label{eq.wbar}
\bar{w}(s) : = \alpha_F^2  \mathbf{1}^T u_0 g_0(s) =  \frac{\sum_i u_{0i}}{\sum_i f_i} g_0(s).
\end{align}
The second term will be denoted by $\tilde{w}(s)$; by Assumption \ref{ass.stable} it has left half-plane poles.
So we have obtained the decomposition
\begin{align}\label{eq.wdecomp}
w(t) = \bar{w}(t) \mathbf{1}+\tilde{w}(t),
\end{align}
intepreted as follows:
\btmz
\item $\bar{w}(t)$ is a \emph{system frequency} term, applied to all nodes;
\item the transient term $\tilde{w}(t)$ represents the individual node deviations from the synchronous response.
\etmz

\subsection{System frequency}

We can obtain more information on the system frequency by observing that since
$\mathbf{1}^T F^\ha V_\perp  = \alpha_F^{-1} v_0^T V_\perp = 0$, we have
\[
(\mathbf{1}^T F) \underbrace{F^{-\ha} V_\perp  \tilde{H}(s)
V_\perp^T  F^{-\ha} u_0}_{\tilde{w}(s)} \equiv 0.
\]
Therefore $ \mathbf{1}^T F w(t) =  \bar{w}(t) \mathbf{1}^T F
\mathbf{1}$ by \reeq{wdecomp}, which gives
\[
\bar{w}(t) = \frac{\sum_i f_i w_i(t)}{\sum_i f_i};
\]
the system frequency is a weighted mean of bus frequencies, in
proportion to their rating. Noting that $m_i = m f_i$, it follows
that $\bar{w}(t)$ is exactly the COI frequency from \reeq{coi}.

Also, returning to \reeq{wbar} we have
\begin{align}\label{eq.wbart}
\bar{w}(t)  = \frac{\sum_i u_{0i}}{\sum_i f_i} g_0(t).
\end{align}
Recall that $g_0(t)$ is the angle impulse response of the
representative machine, or equivalently its angular frequency step
response. Thus $\bar{w}(t)$ corresponds to the frequency observed when exciting the
representative machine (in open loop) with the total system
disturbance normalized by the total scale.

\begin{remark}
Note that this result is \emph{independent of $L$}, i.e. the electrical
network does not affect the time response of the system frequency,
only the machine ratings themselves. Thus, when the network dependent term ($\tilde w$) converges fast to zero,  \eqref{eq.wbart} is a natural candidate for a reduced order model similar to the ones recently considered in~\cite{guggilam2017engineering, apostolopoulou2016balancing}.
\end{remark}


In the following sections we analyze its behavior for the
previously discussed examples.

\subsection{Quantifying the deviation from synchrony}

We now turn our attention to the term $\tilde{w}(t)$ which represents individual bus deviations from a synchronous response. A natural way of quantifying the size of this transient term is through the $\Lt$ norm
\[
\|\tilde{w}\|_2^2 = \int_0^\infty |\tilde{w}(t)|^2 dt.
\]
We now show how this norm can be computed in terms of the parameters of the scaled network Laplacian,
and the impulse response matrix $\tilde{H}(t)=
\diag_{k=1,\ldots,n-1} (h_k(t))$,
Laplace inverse of $\tilde{H}(s)$, which encapsulates all information on the machine model.

\begin{proposition} \label{prop.wtil}
$\|\tilde{w}\|_2^2 = z_0^T Y z_0$, where: \btmz
\item ${Y} \in \R^{(n-1)\times (n-1)}$ is the matrix with elements
\begin{align}\label{eq.defy}
y_{kl} = \gamma_{kl}\ip{h_k}{h_l} &= \gamma_{kl} \int_0^\infty
h_k(t) h_l(t)dt, \\
\mbox{where  }\Gamma = (\gamma_{kl})&:= V_\perp^T F^{-1} V_\perp;
\label{eq.defgamma}
\end{align}
\item $z_0 := V_\perp^T F^{-\ha}u_0 \in \R^{n-1}$.
\etmz
\end{proposition}
\begin{proof}
With the introduced notation we have
\[
\tilde{w}(t) =F^{-\ha} V_\perp  \tilde{H}(t) z_0,
\]
therefore $\tilde{w}(t)^T \tilde{w}(t)= z_0^T  \tilde{H}(t)\Gamma
\tilde{H}(t)z_0.$ The matrix in the above quadratic form has
elements $h_k(t) \gamma_{kl} h_l(t)$, therefore integration in
time yields the result.
\end{proof}

\begin{remark} The metric $\|\tilde{w}\|_2^2$ \emph{does} depend on the
electrical network, through the eigenvalues and eigenvectors of
$L_F$.
\end{remark}

\subsection{Mean synchronization cost for random disturbance
step}

Since the cost discussed above is a function of the disturbance
step $u_0$, it may be useful to find its average over a random
choice of this excitation. Recalling that the components $u_{0i}$
correspond to different buses, it is natural to assume them to be
independent, and thus $ E[u_0 u_0^T] = \Sigma^u$, a diagonal
matrix. Therefore
\[
E[z_0 z_0^T] =  V_\perp^T F^{-1} \Sigma^u V_\perp =:\Sigma^z,
\]
and the expectation for the cost in Proposition \ref{prop.wtil} is
\[
E\left[\|\tilde{w}\|_2^2\right] = E[z_0^T Y z_0] = E[\tr(Y
z_0z_0^T) ] = \tr(Y \Sigma^z).
\]
We look at some special cases: \btmz
\item $\Sigma_u = I$ (uniform disturbances). Then $\Sigma^z = \Gamma$,
and
\[E\left[\|\tilde{w}\|_2^2\right] = \tr(Y\Gamma) = \sum_{k,l} \gamma_{kl}^2 \ip{h_k}{h_l}.\]
\item $\Sigma_u = F$. This means disturbance size follows the square root of the bus rating. Here $\Sigma^z = I$, and
\[E\left[\|\tilde{w}\|_2^2\right] = \tr(Y) = \sum_{k} \gamma_{kk} \|h_k\|_2^2.\]
\item $\Sigma_u = F^2$. This is probably most natural, with disturbances proportional to bus rating. Here $\Sigma^z =  V_\perp^T F V_\perp =
\Gamma^\dagger $ (pseudoinverse); $E\left[\|\tilde{w}\|_2^2\right]
= \tr(Y \Gamma^\dagger)$.
 \etmz

\subsection{The homogeneous case}\label{ssec.homog}

If all machines have the same response $g_0(s)$, setting $F=I$ we
can obtain some simplifications: \btmz
\item $\{\lambda_k\}$ are the eigenvalues of the original
Laplacian $L$.
\item The system frequency is the average $\bar{w}(t)= \frac{1}{n} \sum_i w_i(t)$, and
satisfies
\[
\bar{w}(t)= \frac{1}{n} \left(\sum_i u_{0i}\right) g_0(t).
\]
\item $z_0= V_\perp^T u_0$, and $\Gamma = V_\perp^T  V_\perp = I$. Therefore the matrix $Y$ in
Proposition \ref{prop.wtil} is diagonal, $Y = \diag(\|h_k\|^2)$,
and
\begin{align}\label{eq.wtilhomog}
 \|\tilde{w}\|_2^2 =
\sum_{k=1}^{n-1} (z_{0k})^2 \|h_k\|^2, \end{align} where $z_{0k}=
v_k^T u_0$ is the projection of the excitation vector $u_0$ in the
direction of the $k$-th eigenvector of the Laplacian $L$.
\item The mean synchronization cost for $E[u_0 u_0^T] = I$ (here
all the preceding cases coincide) is
\begin{align}\label{eq.h2norm}
E\left[\|\tilde{w}\|_2^2\right] = \sum_{k=1}^{n-1} \|h_k\|_2^2 =
\|\tilde{H}\|_{\Ht}^2,
\end{align}
the $\Ht$ norm of the transfer function $\tilde{H}(s)$. We recall
that this was obtained by isolating the portion $g_0(s)$
corresponding to the synchronized response (in this case,
projecting onto $\mathbf{1}^\perp$). In this form, the cost
resembles other proposals \cite{bassam,simpson-porco2017}, for the price of synchrony, and \cite{andreasson2017} for the evaluation of the synchronization cost under step changes in homogeneous systems.
\etmz

\section{Application to the swing dynamics}\label{sec.swing}

In this section we assume we are in the situation of Example \ref{ex.swing}, i.e., the
representative machine is
\[
g_0(s) = \frac{1}{s(ms+d)}.
\]
and $m_i = f_i m$, $d_i = f_i d$ are the individual bus
parameters. The corresponding closed loop transfer functions in
\reeq{h} are
\begin{align}\label{eq.hkswing}
h_k(s) &= \frac{1}{ms^2 + ds + \lambda_k}, \quad
 k=0,1,\ldots,n-1.
\end{align}
Note they are stable when $\lambda_k>0$.

\subsection{System frequency}
Inverting the transform and invoking \reeq{wbart} we find that
\begin{align}
\bar{w}(t)  &= \frac{\sum_i u_{0i}}{\sum_i f_i}  \overbrace{\frac{1}{d} \left( 1 - e^{-\frac{d}{m} t}\right)}^{g_0(t)} \\
& = \frac{\sum_i u_{0i}}{\sum_i d_i}  \left( 1 - e^{-\frac{d}{m} t}\right), \quad t>0.
\end{align}
Some comments are in order:
\btmz
\item Again, we recall $\bar{w}(t)$ does not depend on the electrical network.
\item The first-order evolution of $\bar{w}(t)$ implies there is no overshoot; system frequency never deviates to a ``Nadir" further from equilibrium than its steady-state value.
\item The asymptotic frequency $w_\infty = \frac{\sum_i u_{0i}}{\sum_i d_i}$ is the ratio of total disturbance to total damping, but
does not depend on the inertia $m$. The latter only affects the time constant in which this asymptote is achieved.
\item The maximum RoCoF (rate-of-change-of frequency) occurs at $t\to 0+$, and is given by
\begin{align}\label{eq.rocofswing}
\frac{d}{m} \frac{\sum_i u_{0i}}{\sum_i d_i} =  \frac{\sum_i u_{0i}}{\sum_i m_i};
\end{align}
here the total inertia appears, which is natural in a second-order
response to a step in forcing. RoCoF increases for low inertia,
however it need not have a detrimental impact: system frequency
initially varies quickly but never deviates more than $w_\infty$,
independent of $m$. \etmz

\subsection{Synchronization cost}

We now turn to the synchronization cost $\|\tilde{w}\|_2$, which
can be computed by particularizing the result in Proposition
\ref{prop.wtil}. The following result is proved in Appendix \ref{app.ip}.

\begin{proposition}\label{prop.ipswing}
Let $h_k(s)$ be given in  \reeq{hkswing}, and $h_k(t)$ its inverse
transform, for $k=1,\ldots,n-1$. Then:
\begin{align}\label{eq.hipswing}
\ip{h_k}{h_l} =\frac{2 d}{m(\lambda_k - \lambda_l)^2 + 2
(\lambda_k + \lambda_l)d^2}.
\end{align}
\end{proposition}
\mbox{}

It follows that the matrix $Y$ in \reeq{defy} will depend on both
inertia $m$ and damping $d$, so in general both have an impact on
the ``price of synchrony". Note however that inertia only appears in
off-diagonal terms, and the matrix remains bounded as $m\to 0$ or
$m\to \infty$; we thus argue that inertia has limited impact.
We look at this issue in further detail.

\subsubsection{Homogeneous case}
In the case of homogeneous machines, we saw above that $\Gamma=I$
and $Y$ is diagonal, so inertia disappears completely: indeed
using \reeq{wtilhomog} we have
\begin{align}\label{eq.cost2homo}
\|\tilde{w}\|_2^2 = \sum_{k=1}^{n-1} \frac{(v_k^T
u_{0})^2}{2d\lambda_k}.
\end{align}

The cost is inversely proportional to damping, and the direction
of the disturbance $u_0$ also matters. Recalling that $v_{k}$ is
the $k$-th Laplacian eigenvector, the worst-case for a given
magnitude $|u_0|$ is when it is aligned to $v_1$, the Fiedler
eigenvector.

If the disturbance direction is chosen randomly as in Section
\ref{ssec.homog}, then \reeq{h2norm} gives
\begin{align}\label{eq.cost2homorandom}
E\left[\|\tilde{w}\|_2^2\right] = \sum_{k} \|h_k\|_2^2 =
\frac{1}{2 d}\sum_{k}\frac{1}{\lambda_k} = \frac{1}{2 d}
\tr(L^\dagger);
\end{align}
again a similar result to those in \cite{bassam} for homogeneous
systems.

\subsubsection{Heterogeneous, high inertia case}
Assume for this discussion that all the $\lambda_k$ are distinct;
then as $m\to \infty$ we have $y_{kl} \to 0$ for $k\neq l$, so $Y$
again becomes diagonal, and the cost has the limiting expression
\begin{align}\label{eq.cost2heterohigh}
\|\tilde{w}\|_2^2  \stackrel{m\to
\infty}{\longrightarrow}\sum_{k=1}^{n-1} \frac{\gamma_{kk}z_{0k}^2
}{2 d \lambda_k}.
\end{align}
So the high inertia behavior is of a similar structure to the
homogeneous case in \reeq{cost2homo}. Comparisons are not
straightforward, though, since the scaling factor $F$ affects
$z_{0k}, \gamma_{kk}$ and $\lambda_k$ in each of the above terms.

\subsubsection{Heterogeneous, low inertia case}
If $m\to 0$, then the limiting $Y$ matrix is not diagonal. The
corresponding limiting cost is
\begin{align}\label{eq.cost2heterolow}
\|\tilde{w}\|_2^2 \stackrel{m\to 0+}{\longrightarrow}
\sum_{k,l=1}^{n-1} \frac{ \gamma_{kl}z_{0k}z_{0l}}{d (\lambda_k +
\lambda_l)}.
\end{align}
Note, however,  that the diagonal terms are the same as in the
high inertia case. This suggests that inertia plays a limited role
in the $\Lt$ price of synchrony, even in the heterogeneous machine
case.

\section{Model with turbine dynamics} \label{sec.turbine}

We now turn to the model of Example \ref{ex.turbine}, where the
representative machine is
\begin{equation}\label{eq:g_0s-turbine}
g_0(s) = \frac{\tau s + 1}{s \left(m \tau s^2 + (m+d \tau) s+d
+r^{-1}\right)}. \end{equation}

The corresponding closed loop transfer functions in \reeq{h} for $
k=0,1,\ldots,n-1$ are:
\begin{align}
h_k(s) &= \frac{\tau s+1}{m\tau s^3 + (m+d\tau)s^2
+(d+r^{-1}+\lambda_k\tau)s+\lambda_k}. \label{eq.hkturbine}
\end{align}

It can be checked (e.g. by applying the Routh-Hurwitz criterion) that
$h_k(s)$ is stable whenever $\lambda_k>0$.

\subsection{System frequency}

We can again use \reeq{wbart} and \eqref{eq:g_0s-turbine} to compute the system frequency
\begin{align}
\bar{w}(t)  &= \frac{\sum_i u_{0i}}{\sum_i f_i} \mathcal L^{-1}\left\{ g_0(s)\right\},
\end{align}
but now the inverse transform of $g_0(s)$ is more involved.

Using partial fractions we first express
\begin{align*}
g_0(s) &=\frac{1}{d+r^{-1}}
\left(\frac{1}{s} - \frac{s + \left( \frac{1}{\tau}-\frac{r^{-1}}{m}\right)}{s^2 + \left(\frac{1}{\tau} +\frac{d}{m}\right)s + \frac{d+r^{-1}}{m\tau}}  \right)
\end{align*}

The first term provides the steady-state response, which is
\begin{align*}
w_\infty =\frac{\sum_i u_{0i}}{\sum_i f_i}\frac{1}{d+r^{-1}}
=\frac{\sum_i u_{0i}}{d_i+r_i^{-1}};
\end{align*}
this is analogous to the swing equation case, except than the
droop control has been added to the damping. Again, we observe
that inertia plays no role at all in this steady-state deviation.

The transient term is a second-order transfer function, which we
proceed to analyze now.
Its behavior critically depends on whether its poles are real or complex conjugate.
In particular, whenever
\begin{equation}\label{eq:omegad}
 \frac{d+r^{-1}}{m\tau}-\frac{1}{4}\left(\frac{1}{\tau} +\frac{d}{m}\right)^2=:\omega_d^2>0
\end{equation}
the system is under-damped with poles $\eta\pm j\omega_d$, and
\begin{align}
&g_0(t) = \mathcal L^{-1}\left\{\frac{1}{d+r^{-1}}\left(\frac{1}{s}-\frac{s+\gamma}{(s+\eta)^2 + \omega_d^2}\right)\right\}\nonumber\\
&=\frac{1}{d+r^{-1}}\!\left[1\!-\!{e^{-\eta t}}\left(\cos(\omega_d t) \!-\! \frac{(\gamma\!-\!\eta)}{\omega_d}\sin(\omega_d t)\right)\right]
\label{eq:g0-turbine}
\end{align}
where
\begin{equation}\label{eq:eta-gamma}
\eta:= \frac{1}{2}\left(\frac{1}{\tau}+\frac{d}{m}\right)\quad\text{and}\quad \gamma:=\left(\frac{1}{\tau}-\frac{r^{-1}}{m}\right).
\end{equation}

The system frequency time evolution is given by
\begin{equation}\label{eq:wbar-wturbine}
\bar w(t)\!=\!\frac{\sum_iu_{0i}}{\sum_i d_i\!+\!r_i^{-1}}\! \ g_0(t),
\end{equation}
with $g_0(t)$ from \eqref{eq:g0-turbine}.

A few observations are in order:
\begin{itemize}
\item Including the turbine model has a nontrivial effect on the system frequency $\bar w(t)$. It is the presence of the turbine dynamics that provides the characteristic under-damped behavior that produces a Nadir.
\item We have only provided here the solution of $\bar w(t)$ for the (practically more relevant) under-damped
case.
\item Interestingly, \eqref{eq:omegad} shows that the system may become over-damped by either increasing $m$, or decreasing
 $m$! However, the behavior is different for each case: in the very high
 inertia case the Nadir disappears; whereas when $m$ goes to zero, there is an overshoot in the overdamped response.
Since in practice this occurs only for very low inertia and already way beyond the acceptable deviation, we are justified in our focus on the under-damped case.
 \end{itemize}
We now proceed to compute the Nadir and RoCoF for this situation.

\subsubsection{Nadir} \mbox{}\\
In order to compute the Nadir we will use
\[
||\bar w||_\infty =\left|\frac{\sum_i u_{0i}}{\sum_i f_i}\right| ||g_0||_\infty.
\]
Thus, one can compute the Nadir by finding the maximum excursion of $g_0(t)$.
The following proposition summarizes the overall calculation which can be find in Appendix \ref{app.nadir}.

\begin{proposition}[Nadir]\label{prop.nadir}
Given a power system under Assumption \ref{ass.scale} with generators containing first order turbine dynamics ($g_i(s)$ given by~\eqref{eq.giturbine}). Then under the under-damped condition \eqref{eq:omegad}, the Nadir is given by
\begin{equation}\label{eq:nadir}
||\bar w||_\infty =\frac{\left|\sum_i u_{0i}\right|}{\sum_i f_i} \frac{1}{d\!+\!r^{-1}}\left(1\!+\!\sqrt{\frac{\tau r^{-1}}{m} }e^{\!-\!\frac{\eta}{\omega_d}\left(\!\phi+\frac{\pi}{2}\!\right)}\right),
\end{equation}
where the phase  $\phi\in(-\frac{\pi}{2},\frac{\pi}{2})$ is
uniquely determined by
\begin{equation}\label{eq:sinphi}
\sin(\phi)=\frac{\left(\frac{1}{\tau
}-\eta\right)}{\sqrt{\omega_d^2 + \left(\frac{1}{\tau}-\eta\right)^2}}=\frac{m-d\tau}{2\sqrt{m\tau r^{-1}}}.\vspace{1ex}
\end{equation}
\end{proposition}

The dependence of \eqref{eq:nadir} on $m$ is not
straightforward, as $\phi$, $\eta$, and $\omega_d$ depend on it.
The next proposition shows that the dependence is as expected by
conventional power engineering wisdom.
\begin{proposition}\label{prop.dNadirdm}
Given a power system under Assumption \ref{ass.scale} with generators containing first order turbine dynamics ($g_i(s)$ given by~\eqref{eq.giturbine}). Then under the under-damped condition \eqref{eq:omegad}, the maximum frequency deviation $||\bar w||_\infty$ is a decreasing function of $m$, i.e., $\frac{\partial}{\partial m}||\bar w||_\infty<0$.
\end{proposition}
The proof can be found in Appendix \ref{app.nadir.2}.

\subsubsection{RoCoF}\mbox{}\\
A similar procedure as the one used to study the Nadir of the system frequency can be used to investigate the properties of the maximum rate of change of frequency (RoCoF).

\begin{proposition}[RoCoF]\label{prop.rocof}
Given a power system under Assumption \ref{ass.scale} with generators containing first order turbine dynamics ($g_i(s)$ given by~\eqref{eq.giturbine}). Then under the under-damped condition \eqref{eq:omegad}, the RoCoF is given by
\begin{equation}\label{eq:rocof}
||\dot{\bar w}||_\infty = \frac{\left|
\sum_iu_{0,i}\right|}{\sum_i f_i}\frac{1}{m}.
\end{equation}
\end{proposition}

The proof of Proposition \ref{prop.rocof} can also be found in  Appendix \ref{app.rocof};
the main difference with Proposition \ref{prop.nadir} is that, while not trivial to establish, here
 the maximum is always achieved at $t=0+$, exactly as in the
second order case of \reeq{rocofswing}.

The dependence of $||\dot{\bar w}||_\infty$ on $m$ is now easily
addressed and again as expected: RoCoF decreases with $m$.

\subsection{Synchronization cost}
The synchronization cost $\|\tilde{w}\|^2$ can once again be computed
through  Proposition \ref{prop.wtil}, which requires finding the
inner products $\ip{h_k}{h_l}$, in this case for the functions in
\reeq{hkturbine}.

Since the corresponding expression is in general rather unwieldy (see Appendix \ref{app.ip}), we
will present some simpler cases, beginning with $k=l$; the norm, found in
Appendix \ref{app.ip} is:
\begin{align}\label{eq.hknormturbine}
\|h_k\|_2^2 = \frac{m+\tau(\lambda_k \tau +
d)}{2\lambda_k\left[m(r^{-1}+d) +\tau d(r^{-1}+\lambda_k \tau +
d)\right]}.
\end{align}

\subsubsection{Homogeneous case}

The above expression suffices to analyze the case of homogeneous machines,
where $\Gamma=I$ and $Y$ is diagonal. We have from \reeq{wtilhomog} that
\begin{align*}
\|\tilde{w}\|_2^2 = \sum_{k=1}^{n-1} (v_k^T u_{0})^2 \|h_k\|_2^2;
\end{align*}
from \reeq{hknormturbine} we see that, in contrast to the second
order machine model, the inertia $m$ does affect the
synchronization cost. A closer look at $\|h_k\|^2$ as a (linear
fractional) function of $m$ shows that it is \emph{decreasing} in
$m \in (0,\infty) $, going from
\[
\|h_k\|^2_2 \stackrel{m\to 0+}{\longrightarrow} \frac{1}{2\lambda_k d } \cdot \frac{\lambda_k \tau + d}{r^{-1}+\lambda_k \tau + d},
\]
to
\[
\|h_k\|^2_2 \stackrel{m\to \infty}{\longrightarrow}  = \frac{1}{2\lambda_k d } \cdot \frac{d}{r^{-1}+d}.
\]
So higher inertia is beneficial in this case. Recalling that the corresponding cost for the swing dynamics is
$\frac{1}{2\lambda_k d }$, we see that this cost has been reduced. In the high inertia case,
the main change is the increased damping through the droop coefficient $r^{-1}$.

\subsubsection{Heterogeneous, high inertia case}

As mentioned, the formula for $\ip{h_k}{h_l}$ for $k\neq l$ is quite formidable, but
we can give its approximation in the limit of large $m$:
\begin{align*}
\ip{h_k}{h_l} \stackrel{m\to \infty}{\sim}
\frac{2(d+r^{-1})}{m(\lambda_k-\lambda_l)^2}, \quad k\neq l.
\end{align*}
This assumes $\lambda_k\neq \lambda_l$. So if the eigenvalues of
the scaled Laplacian $L_F$ are distinct, we see that again the
matrix $Y$ becomes diagonal as $m\to \infty$. The limiting cost is
\[
\|\tilde{w}\|_2^2 \stackrel{m\to \infty}{\longrightarrow}
\sum_{k=1}^{n-1} \frac{z_{0k}^2 \gamma_{kk}}{2\lambda_k d } \cdot
\frac{d}{r^{-1}+d}.
\]
This expression amounts to reducing to the cost
\reeq{cost2heterohigh} for the second order dynamics, by the
fraction $\frac{d}{r^{-1}+d}$. So the role of the turbine in a
high inertia system is again mainly  a change in the droop
coefficient.

\subsubsection{Heterogeneous, low inertia case}
In the low inertia limit, we find that $\ip{h_k}{h_l} \stackrel{m\to 0}{\longrightarrow} \frac{N}{D}$, where
\begin{align*}
N =\ &  2d (d+r^{-1})+ \tau(2d+r^{-1})(\lambda_k+\lambda_l) + 2 \lambda_k\lambda_l \tau^2,\\
D =\ & 2d (d+r^{-1})^2(\lambda_k+\lambda_l) + d \tau(2d+r^{-1})(\lambda_k+\lambda_l)^2 \\ & +  2 d \tau \lambda_k\lambda_l [2r^{-1}+\tau (\lambda_k+\lambda_l)].
\end{align*}
So the limiting matrix $Y$ is not diagonal, as in the second order
case; an expression analogous to \reeq{cost2heterolow} can be
written. Comparisons between the two are not straightforward here,
and must be pursued by numerical experimentation.

\section{Conclusions}\label{sec.concl}

We have studied system-theoretic measures of synchronization
performance, with the aim of covering more realistic scenarios
than the recent literature, and also closing the gap with power
engineering practice. In particular, for a family of heterogeneous
machine systems, we have focused on the step response of the bus
frequency vector, decomposed as a system-wide weighted mean and the
vector of relative differences to it. The key mathematical assumption is a
relative proportionality in machine parameters, which may be
incorporated together with the network model to perform an
adequate diagonalization.

With this transformation, the natural system frequency (motion of
the COI) becomes independent of the network, and its
characteristics can be studied through a representative machine.
The energy of the synchronization error around this mean depends
on both network and machine models, but we have an expression that
encapsulates the latter in terms of a matrix of inner products.
From this general result special cases can be studied.

A key question of interest to practitioners is the role of
inertia, in particular whether low inertia can compromise
performance. Our analysis shows that if a second-order, swing
equation model is used for each machine, the impact of inertia is
small. The global system frequency  exhibits no overshoot, inertia
affecting mainly its time constant; while inertia appears in the
energy of oscillations, its impact is not significant. The story
changes if a dynamic model of turbine control is adopted, where
the turbine time-constant is in the order of magnitude of the
swing dynamics. In that case, inertia does play a positive role,
providing resilience of the peak system frequency deviation and
reducing the norm of relative oscillations.

Future work will involve carrying out these calculations for
actual networks with real parameters.

\appendices

\section{Inner product computation}\label{app.ip}

We show here how to evaluate the inner product in \reeq{defy}
using state-space methods. We start with a state-space
realization of the representative machine,
\begin{align*}
g_0(s) = \stsp{A}{B}{C}{0}.
\end{align*}
If follows easily that each $h_k$ in \reeq{h} has realization
\begin{align*}
h_k(s) = \stsp{A_k}{B}{C}{0}, \mbox{ where } A_k = A - \lambda_k B C.
\end{align*}
Note that the state matrix $A_k$ is the only one that depends on
the eigenvalue $\lambda_k$ under consideration. By Assumption
\ref{ass.stable}, $A_k$ is a Hurwitz matrix for any $\lambda_k>0$.

Writing $h_k(t) = C e^{A_k t} B$ for the impulse response, we compute the inner product between two such functions:
\begin{align}
\ip{h_k}{h_l} &= \int_0^\infty h_k(t) h_l(t)^T dt \nonumber  \\
&= C \underbrace{
\left(\int_0^\infty e^{A_k t} B B^T e^{A_l^T t} dt \right)}_{Q_{kl}} C^T; \label{eq.Qkl}
\end{align}
here $^T$ denotes matrix transpose. A standard calculation shows that
 $Q_{kl}$ satisfies the Sylvester equation
\begin{align}\label{eq.sylv}
A_k Q_{kl} + Q_{kl} A_l^T + BB^T = 0.
\end{align}
Furthermore since the eigenvalues of $A_k, A_l$ never add up to zero it follows (see
\cite{ZhouDG}) that \reeq{sylv} has a unique solution $Q_{kl}$.
Thus, the relevant inner product can be found by solving the above linear equation and
substituting into \reeq{Qkl}.

\subsection*{Second order machine model}

In the situation of Section \ref{sec.swing}, it is easily checked that
\[
A_k = \bmat{0 & 1 \\ -\frac{\lambda_k}{m} & -\frac{d}{m}} \quad B= \bmat{0 \\ -\frac{1}{m}}; \quad C = \bmat{1 & 0 }.
\]
In this case the solution to the Sylvester equation is
\begin{align*}
Q_{kl} = \frac{2 d}{m(\lambda_k - \lambda_l)^2 + 2(\lambda_k + \lambda_l)d^2}
\bmat{1 & \frac{\lambda_k - \lambda_l}{2d} \\ \frac{\lambda_l - \lambda_k}{2d}&  \frac{\lambda_k + \lambda_l}{2m}};
\end{align*}
Substitution into \reeq{Qkl} for the given $C$ proves Proposition
\ref{prop.ipswing}.

\subsection*{Third order machine model}

Here the relevant matrices are
\[ \setlength{\arraycolsep}{.3em}
A_k = \bmat{0 & 1 & 0\\-\frac{\lambda_k}{m} & -\frac{d}{m} & \frac{1}{m}\\
0 &  -\frac{r^{-1}}{\tau} &  -\frac{1}{\tau}};\
B = \bmat{0\\-\frac{1}{m}\\0};
\ C = \bmat{1 & 0 & 0}.
\]
The Sylvester equations for $Q_{kl}$ in this case (9 linear equations, 9 unknowns) give
unwieldy expressions. We report first the case $k=l$, which remains tractable; here
we have a  Lyapunov equation with symmetric solution
\begin{align}
Q_{kk} =
\frac{1}{\Delta} \bmat{\frac{m + \tau(\lambda_k + d)}{\lambda_k} & 0 & -\tau r^{-1} \\
0 & \frac{m+\tau(r^{-1}+ \lambda_k \tau + d)}{m} & -r^{-1}\\
-\tau r^{-1} &   r^{-1} & r^{-2}},
\end{align}
where $\Delta = 2[m(r^{-1}+d) +  \tau d (r^{-1}+d + \lambda_k\tau)].$
By looking at the $(1,1)$ element of this matrix we find the norm
$\|h_k\|^2 = \ip{h_k}{h_k}$, which coincides with the expression given in
\reeq{hknormturbine}.

Going now to the general case $k\neq l$, we report here only the inner product 
obtained from the (1,1) element of $Q_{kl}$, itself found by solving the Sylvester equation
using the Matlab symbolic toolbox:

\begin{align*}
\langle{h_k},{h_l}\rangle= \frac{N}{D},
\end{align*}
where
\begin{align*}
N=&
2(2dm^2 + 2m^2r^{-1} + 2d^3\tau^2 + 4d^2m\tau + 2d^2\lambda_k\tau^3 \\
& + 2d^2\lambda_l\tau^3 + 2d^2r^{-1}\tau^2 + 4dmr^{-1}\tau + 2d\lambda_k\lambda_l\tau^4 \\
&+ 2d\lambda_km\tau^2 + 2d\lambda_lm\tau^2 + d\lambda_kr^{-1}\tau^3 + d\lambda_lr^{-1}\tau^3 \\
& + \lambda_kmr^{-1}\tau^2 + \lambda_lmr^{-1}\tau^2), \\
D =&
4d^4\lambda_k\tau^2 + 4d^4\lambda_l\tau^2 + 4d^3\lambda_k^2\tau^3 + 8d^3\lambda_k\lambda_l\tau^3 \\
& + 8d^3\lambda_km\tau + 8d^3\lambda_kr^{-1}\tau^2 + 4d^3\lambda_l^2\tau^3 + 8d^3\lambda_lm\tau \\
& + 8d^3\lambda_lr^{-1}\tau^2  + 4d^2\lambda_k^2\lambda_l\tau^4 + 6d^2\lambda_k^2m\tau^2 \\
& + 2d^2\lambda_k^2r^{-1}\tau^3 + 4d^2\lambda_k\lambda_l^2\tau^4 + 4d^2\lambda_k\lambda_lm\tau^2 \\
&+ 12d^2\lambda_k\lambda_lr^{-1}\tau^3  + 4d^2\lambda_km^2 + 16d^2\lambda_kmr^{-1}\tau\\
& + 4d^2\lambda_kr^{-2}\tau^2+ 6d^2\lambda_l^2m\tau^2 + 2d^2\lambda_l^2r^{-1}\tau^3 \\
&+ 4d^2\lambda_lm^2  + 16d^2\lambda_lmr^{-1}\tau + 4d^2\lambda_lr^{-2}\tau^2 +2d\lambda_k^3m\tau^3 \\
& - 2d\lambda_k^2\lambda_lm\tau^3 + 4d\lambda_k^2m^2\tau - 2d\lambda_k\lambda_l^2m\tau^3 - 8d\lambda_k\lambda_lm^2\tau \\
& + 16d\lambda_k\lambda_lmr^{-1}\tau^2  + 8d\lambda_km^2r^{-1}  +  8d\lambda_kmr^{-2}\tau \\
& + 2d\lambda_l^3m\tau^3 + 4d\lambda_l^2m^2\tau + 8d\lambda_lm^2r^{-1} + 8d\lambda_lmr^{-2}\tau \\ 
& + 2\lambda_k^3\lambda_lm\tau^4 + 2\lambda_k^3m^2\tau^2 + \lambda_k^3mr^{-1}\tau^3 - 4\lambda_k^2\lambda_l^2m\tau^4 \\
 &- 2\lambda_k^2\lambda_lm^2\tau^2 - \lambda_k^2\lambda_lmr^{-1}\tau^3  + 2\lambda_k^2m^3 - 2\lambda_k^2m^2r^{-1}\tau  \\
 &+ 2\lambda_k\lambda_l^3m\tau^4 - 2\lambda_k\lambda_l^2m^2\tau^2  - \lambda_k\lambda_l^2mr^{-1}\tau^3 \\
 &- 4\lambda_k\lambda_lm^3 + 4\lambda_k\lambda_lm^2r^{-1}\tau  + 4\lambda_km^2r^{-2} + 2\lambda_l^3m^2\tau^2 \\ &+ \lambda_l^3mr^{-1}\tau^3  + 2\lambda_l^2m^3 - 2\lambda_l^2m^2r^{-1}\tau+ 4\lambda_lm^2r^{-2}.
\end{align*}
The limiting cases $m\to 0$ and $m\to \infty$, presented in Section \ref{sec.turbine} were obtained from this general formula.

\section{Proof of Proposition \ref{prop.nadir}}\label{app.nadir}

\begin{proof}

Since the second order system $g_0(s)$ is stable, the maximum of the impulse response is either at $t=0$ or the first time $\dot{g_0}(t)=0$. Since $\bar g_0(0)=0$, then it must be the latter.

Therefore we need to find the time instance $t_\text{Nadir}$ such that $|g_0(t_\text{Nadir})|=\underset{t\geq0}{\sup}|g_0(t)|$.

The time derivative of \eqref{eq:g0-turbine} is given by
\begin{align}
\dot g_0(t) &=\mathcal L^{-1}\left\{ sg_0(s)-g_0(t)|_{t=0^+}\right\} \\
&= \frac{1}{m}\sqrt{1+(\tan(\phi))^2}e^{-\eta t}
\cos(\omega_{d}t-\phi)~\label{eq:dot-g0-turbine}
\end{align}
where $\phi$ is defined as in the \eqref{eq:sinphi}.

Setting now $\dot g_0(t)=0$ in \eqref{eq:dot-g0-turbine} gives
\[
t_k =  \frac{\phi +\frac{\pi}{2}+k\pi}{\omega_d}, \quad k\geq0
\]
and since $\phi+\frac{\pi}{2}>0$, the first maximum is for $k=0$.
Therefore $$t_\text{Nadir}=\frac{\phi+\frac{\pi}{2}}{\omega_d}$$
which after substituting in \eqref{eq:g0-turbine} gives
\begin{align*}
&||g_0||_\infty = |g_0(t_\text{Nadir})| = \\
&\!=\!\frac{1}{d\!+\!r^{-1}}\!
\left[
1 \!-\!e^{\!-\!\frac{\eta}{\omega_d}\left(\!\phi+\frac{\pi}{2}\!\right)}\!\!
\left(\! \cos\!\left(\phi\!+\!\frac{\pi}{2}\right) \!+\! \frac{\gamma\!-\!\eta}{\omega_d} \sin\!\left(\phi\!+\!\frac{\pi}{2}\right) \!\right)\!
\right]\\
&\!=\!\frac{1}{d\!+\!r^{-1}}\left(1+\sqrt{\frac{\tau r^{-1}}{m} }e^{\!-\!\frac{\eta}{\omega_d}\left(\!\phi+\frac{\pi}{2}\!\right)}\right)
\end{align*}
where the last step follows from \eqref{eq:omegad}, \eqref{eq:eta-gamma}, \eqref{eq:sinphi} and 
\begin{equation}\label{eq:cosphi}
\cos(\phi) = \frac{\omega_d}{ \sqrt{\omega_d^2 + \left(\frac{1}{\tau}-\eta\right)^2 } }=\sqrt{1-\frac{(m-d\tau)^2}{4m\tau r^{-1}}},
\end{equation}
as well as several trigonometric identities.

Therefore the nadir of the system frequency is given by
\begin{equation}\label{eq:nadir-turbine}
||\bar w||_\infty = \frac{\left|\sum_i u_i \right|}{\sum_i f_i}\frac{1}{d+r^{-1}}\left(1+\sqrt{\frac{\tau r^{-1}}{m} }e^{\!-\!\frac{\eta}{\omega_d}\left(\!\phi+\frac{\pi}{2}\!\right)}\right)
\end{equation}
\end{proof}

\section{Proof of Proposition \ref{prop.dNadirdm}}\label{app.nadir.2}

The proof of this proposition requires the following two lemmas.

\begin{lemma}\label{lem:dphidm}
Given a power system under Assumption \ref{ass.scale} with $g_i(s)$ given by~\eqref{eq.giturbine}, the derivative of $\phi$ with respect to $m$ is given by
\[
\frac{\partial \phi}{\partial m}=\frac{m+d\tau}{2m\sqrt{4r^{-1}m\tau-(m-d\tau)^2}}
\]
\end{lemma}
\begin{proof}
Notice that while it is not possible to derive a closed form condition for $\phi$ in terms of the system parameters without the aid of a trigonometric function, it is possible to achieve such expression for $\frac{\partial \phi}{\partial m}$. This is achieved by first computing
\begin{align*}
\frac{\partial}{\partial m}\tan(\phi(m))&=\frac{\partial}{\partial m}\left( \frac{(m\!-\!d\tau)}{\sqrt{4r^{\!-\!1}m\tau\!-\!(m\!-\!d\tau)^2}} \right) \\
&=\frac{2r^{\!-\!1}\tau(m+d\tau)}{(4r^{\!-\!1}m\tau\!-\!(m\!-\!d\tau)^2)^\frac{3}{2}}
\end{align*}

Now using the fact that $\phi\in(-\frac{\pi}{2},\frac{\pi}{2})$ we get
\begin{align*}
\frac{\partial }{\partial m}\phi(m) &= \frac{\partial}{\partial m}\arctan\left(\tan (\phi(m))\right)\\
&=\left(\frac{\partial}{\partial x}\arctan(x)\big|_{x=\tan(\phi)}\right)
\frac{\partial}{\partial m}\tan(\phi(m))\\
&={\cos(\phi)^2} \frac{\partial}{\partial m}\tan(\phi(m))\\
&=\left(1-\frac{(m-d\tau)^2}{4r^{-1}m\tau}\right)\frac{\partial}{\partial m}\tan(\phi(m))\\
&=\left(1-\frac{(m-d\tau)^2}{4r^{-1}m\tau}\right)
\frac{2r^{\!-\!1}\tau(m+d\tau)}{(4r^{\!-\!1}m\tau\!-\!(m\!-\!d\tau)^2)^\frac{3}{2}}\\
&=\frac{m+d\tau}{2m\sqrt{4r^{-1}m\tau-(m-d\tau)^2}}
\end{align*}
\end{proof}

\begin{lemma}\label{lem:detaomega-dm}
Given a power system under Assumption \ref{ass.scale} with $g_i(s)$ given by~\eqref{eq.giturbine}, the derivative of $\frac{\eta}{\omega_d}$ with respect to $m$ is given by
\[
\frac{\partial}{\partial m}\left( \frac{\eta}{\omega_d}\right)
=
\frac{
2\tau(d+ r^{-1})(m-d\tau)
}{
({4m\tau r^{-1}-\left(m-d\tau\right)^2})^\frac{3}{2}
}
\]
\end{lemma}

\begin{proof}
The proof is just by direct computation
\begin{align*}
&\frac{\partial}{\partial m}\left( \frac{\eta}{\omega_d}\right)
=\frac{\partial}{\partial m}\left(\frac{m+d\tau}{\sqrt{4m\tau r^{\!-\!1}\!-\!\left(m\!-\!d\tau\right)^2}}\right)\\
&=
\frac{
{4m\tau r^{\!-\!1}\!-\!\left(m\!-\!d\tau\right)^2}
\!-\!(2\tau r^{\!-\!1}\!-\!m+d\tau)(m+d\tau)
}{
\sqrt{(4m\tau r^{\!-\!1}\!-\!\left(m\!-\!d\tau\right)^2)}({4m\tau r^{\!-\!1}\!-\!\left(m\!-\!d\tau\right)^2})
}\\
&=
\frac{
4m\tau r^{\!-\!1} \!-\! 2\tau r^{\!-\!1}(m+d\tau) \!-\! (m\!-\!d\tau)^2+m^2 \!-\! (d\tau)^2
}{
({4m\tau r^{\!-\!1}\!-\!\left(m\!-\!d\tau\right)^2})^\frac{3}{2}
}\\
&=
\frac{
2m\tau r^{\!-\!1} \!-\! 2(r^{\!-\!1}\tau)(d\tau) +2md\tau \!-\! 2(d\tau)^2
}{
({4m\tau r^{\!-\!1}\!-\!\left(m\!-\!d\tau\right)^2})^\frac{3}{2}
}\\
&=
\frac{
2m\tau(d+ r^{\!-\!1}) \!-\! 2(d\tau)\tau(d+r^{\!-\!1})
}{
({4m\tau r^{\!-\!1}\!-\!\left(m\!-\!d\tau\right)^2})^\frac{3}{2}
}\\
&=
\frac{
2\tau(d+ r^{\!-\!1})(m\!-\!d\tau)
}{
({4m\tau r^{\!-\!1}\!-\!\left(m\!-\!d\tau\right)^2})^\frac{3}{2}
}
\end{align*}
\end{proof}

We are now ready to prove Proposition \ref{prop.dNadirdm}.

\begin{proof}
Using lemmas \ref{lem:dphidm} and \ref{lem:detaomega-dm}, and $\alpha=\frac{\pi}{2}+\phi$ we can compute
\begin{align*}
&\frac{\partial }{\partial m}\left[\left(\frac{d\!+\!r^{\!-\!1}}{\sqrt{\tau r^{\!-\!1}}}\right)||g_0||_\infty\right]=
\frac{\partial }{\partial m}\left[
\left(\frac{1}{\sqrt{\tau r^{\!-\!1}}}\!+\!
\frac{1}{\sqrt{m}}
e^{\!-\!\frac{\eta}{\omega_d}\alpha}\right)
\right]\\
&=e^{\!-\!\frac{\eta}{\omega_d}\alpha}
\left(
\frac{\partial}{\partial m}\left(\frac{1}{\sqrt{m}}\right)
\!-\!
\frac{1}{\sqrt{m}}
\left(
\frac{\eta}{\omega_d}\frac{\partial \phi}{\partial m} \!+\!\alpha\frac{\partial}{\partial m}\left(\frac{\eta}{\omega_d}\right)
\right)
\right)\\
&=e^{\!-\!\frac{\eta}{\omega_d}\alpha}
\left(
\left(\frac{\!-\!\frac{1}{2}m^{\!-\!\frac{1}{2}}}{{m}}\right)
\!-\!
\frac{1}{\sqrt{m}}
\left(
\frac{\eta}{\omega_d}\frac{\partial \phi}{\partial m} \!+\! \alpha\frac{\partial}{\partial m}\left(\frac{\eta}{\omega_d}\right)
\right)
\right)\\
&=\frac{1}{d\!+\!r^{\!-\!1}}\sqrt{\frac{\tau r^{\!-\!1}}{m}}
\frac{e^{\!-\!\frac{\eta}{\omega_d}\alpha}}{2}
\!\left(\!
\left(\frac{\!-\!1}{{m}}\right)
\!-\!
2
\left(\!
\frac{\eta}{\omega_d}\frac{\partial \phi}{\partial m} \!+\! \alpha\frac{\partial}{\partial m}\left(\!\frac{\eta}{\omega_d}\!\right)
\!\right)
\!\right)\\
&=
\frac{\sqrt{\tau r^{-1}}e^{\!-\!\frac{\eta}{\omega_d}\alpha}}
{2\sqrt{m}(d\!+\!r^{\!-\!1})}
\left(
\left(\frac{\!-\!1}{{m}}\right)
\!-\!
2
\left(
\frac{\eta}{\omega_d}\frac{m\!+\!d\tau}{2m\sqrt{4r^{\!-\!1}m\tau\!-\!(m\!-\!d\tau)^2}} \! \right.\right.\\
&\qquad\left.\left.+\! \alpha\frac{2\tau(d\!+\! r^{\!-\!1})(m\!-\!d\tau)}{({4m\tau r^{\!-\!1}\!-\!\left(m\!-\!d\tau\right)^2})^\frac{3}{2}}
\right)
\right)\\
&=\frac{\sqrt{\tau r^{-1}}e^{\!-\!\frac{\eta}{\omega_d}\alpha}}
{2m\sqrt{m}(d\!+\!r^{\!-\!1})}
\left(\!-\!1 \!-\!
\left(
\frac{\eta}{\omega_d}\frac{m\!+\!d\tau}{\sqrt{4r^{\!-\!1}m\tau\!-\!(m\!-\!d\tau)^2}} \right.\right.\\
&\qquad\left.\left. +\! 2m\alpha
\frac{2\tau(d\!+\! r^{\!-\!1})(m\!-\!d\tau)}{({4m\tau r^{\!-\!1}\!-\!\left(m\!-\!d\tau\right)^2})^\frac{3}{2}}
\right)
\right)\\
&=\frac{1}{d\!+\!r^{\!-\!1}}\sqrt{\frac{\tau r^{\!-\!1}}{m}}
\frac{e^{\!-\!\frac{\eta}{\omega_d}\alpha}}{2m}
\Bigg(\!-\!1  \!-\!\Bigg(\frac{\left(m\!+\!d\tau\right)}{\sqrt{4m\tau r^{\!-\!1}\!-\!\left(m\!-\!d\tau\right)^2}}
\\
&\quad\frac{m\!+\!d\tau}{\sqrt{4r^{\!-\!1}m\tau\!-\!(m\!-\!d\tau)^2}} \!+\! 2m\alpha
\frac{2\tau(d\!+\! r^{\!-\!1})(m\!-\!d\tau)}{({4m\tau r^{\!-\!1}\!-\!\left(m\!-\!d\tau\right)^2})^\frac{3}{2}}
\Bigg)
\Bigg)\\
&=\frac{1}{d\!+\!r^{\!-\!1}}\sqrt{\frac{\tau r^{\!-\!1}}{m}}
\frac{e^{\!-\!\frac{\eta}{\omega_d}\alpha}}{2m} \Bigg(\!-\!1 \!-\!
\frac{(m\!+\!d\tau)^2}
{{4m\tau r^{\!-\!1}\!-\!(m\!-\!d\tau)^2}}\\
&\quad \!-\!\alpha
\frac{4m\tau(d\!+\! r^{\!-\!1})(m\!-\!d\tau)}{({4m\tau r^{\!-\!1}\!-\!(m\!-\!d\tau)^2})^\frac{3}{2}}
\Bigg)<0\\
\end{align*}
\end{proof}

\section{Proof of Proposition \ref{prop.rocof}}\label{app.rocof}

The poof of this proposition requires the computation of the following lemma.
\begin{lemma}
Given $g_0(s)$ as in \eqref{eq.giturbine}
\[
\ddot{g}_0(t) = -\frac{d}{m^2}e^{-\eta t}\frac{\cos(\omega_d t-\beta)}{\cos(\beta)}
\]
where the angle $\beta\in(-\frac{\pi}{2},\frac{\pi}{2})$, $\beta>\phi$ for $\phi\in(-\frac{\pi}{2},\frac{\pi}{2})$ defined by \eqref{eq:cosphi} and \eqref{eq:sinphi},  and $\beta$ is uniquely defined by
\[
\cos(\beta) =\frac{\omega_d}{ \sqrt{\omega_d^2+\left(\frac{1}{\tau}-\eta+\frac{r^{-1}}{d\tau}\right)^2} }
\]
and
\[
\sin(\beta) =\frac{\frac{1}{\tau}-\eta+\frac{r^{-1}}{d\tau}}{ \sqrt{\omega_d^2+\left(\frac{1}{\tau}-\eta+\frac{r^{-1}}{d\tau}\right)^2} }.
\]
\end{lemma}
\begin{proof}
We first compute
\begin{align}
s^2g_0(s)&=\frac{1}{m\tau}\frac{\tau s^2 + s}{s^2+\left(\frac{1}{\tau}+\frac{d}{m}\right)s+\frac{d+r^{-1}}{m\tau}}\\
&=\frac{1}{m\tau}\left(\tau - \frac{\frac{d\tau}{m}s+\frac{d+r^{-1}}{m}}{s^2+\left(\frac{1}{\tau}+\frac{d}{m}\right)s+\frac{d+r^{-1}}{m\tau}}\right)\\
&=\frac{1}{m\tau}\left(\tau - h(s)\right)
\end{align}
where
\[
h(s) =\frac{\frac{d\tau}{m}s+\frac{d+r^{-1}}{m}}{s^2+\left(\frac{1}{\tau}+\frac{d}{m}\right)s+\frac{d+r^{-1}}{m\tau}}.
\]

Thus we can compute
\begin{align*}
\ddot g_0(t)&=\mathcal L^{-1}\left[ s^2g_0(s)-sg_0(t)|_{t=0^+}-\dot g_0(t)|_{t=0^+}\right]\\
&=\mathcal L^{-1}\left[ \frac{1}{m\tau}(\tau-h(s)) - \frac{1}{m}\right]\\
&=-\frac{1}{m\tau}\mathcal L^{-1}\left[h(s)\right]
\end{align*}

Therefore, it is enough to compute
\begin{align*}
&h(t)=\mathcal L^{\!-\!1}\left[ h(s)\right]\\
&=\mathcal L^{\!-\!1}\left[\frac{\frac{d\tau}{m}s\!+\!\frac{d\!+\!r^{\!-\!1}}{m}}{s^2\!+\!\left(\frac{1}{\tau}\!+\!\frac{d}{m}\right)s\!+\!\frac{d\!+\!r^{\!-\!1}}{m\tau}}\right]\\
&= \mathcal L^{\!-\!1}\left[
\frac{\frac{d\tau}{m}s\!+\!\frac{d\!+\!r^{\!-\!1}}{m}}{ (s\!+\!\eta)^2 \!+\!\omega_d^2}
\right]\\
&=e^{\!-\!\eta t}\left(\frac{d\tau}{m}\cos(\omega_d t) \!+\! \frac{\frac{d\!+\!r^{\!-\!1}}{m}\!-\!\eta\frac{d\tau}{m}}{\omega_d}\sin(\omega_d t)\right)\\
&=\frac{d\tau}{m\omega_d}e^{\!-\!\eta t}\left( \omega_d \cos(\omega_dt)\!+\!\left( \frac{1}{\tau}\!-\!\eta \!+\!\frac{r^{\!-\!1}}{d\tau}\right)\sin(\omega_dt)\right)\\
&=\frac{d\tau}{m\omega_d}e^{\!-\!\eta t}\left( \omega_d \cos(\omega_dt)\!+\!\left( \frac{1}{\tau}\!-\!\eta \!+\!\frac{r^{\!-\!1}}{d\tau}\right)\sin(\omega_dt)\right)\\
&=\frac{d\tau}{m}e^{\!-\!\eta t}\frac{1}{\cos(\beta)}\left(\cos(\beta)\cos(\omega_dt)\!+\!\sin(\beta)\sin(\omega_dt)\right)\\
&=\frac{d\tau}{m}e^{\!-\!\eta t}\frac{\cos(\omega_d t\!-\!\beta)}{\cos(\beta)}
\end{align*}

\end{proof}

We are now ready to prove Proposition \ref{prop.rocof}

\begin{proof}
When $t=0^+$, $\ddot g_0(0)<0$. Therefore, the initial trend the $\dot{g_0}(t)$ is decreasing and therefore
\[
\dot{g_0}(0)=
\frac{1}{m}\sqrt{1+(\tan(\phi))^2}e^{0}\cos(0-\phi) = \frac{1}{m}
\]
is a local maximum.

We will now show that this is indeed a global maximum.
The function $\ddot g_0(t)$ crosses zero for the first time when $\omega_d t-\beta=\frac{\pi}{2}$ or equivalently
\begin{equation}\label{eq:t_star}
t^* = \frac{\beta+\frac{\pi}{2}}{\omega_d}
\end{equation}

By substituting \eqref{eq:t_star} into \eqref{eq:dot-g0-turbine} and
defining
\begin{align*}
\Delta&=\sqrt{\omega_d^2 + \left(\frac{1}{\tau}\!-\!\eta+\frac{r^{-1}}{d\tau}\right)^2}
\end{align*}
we get
\begin{align*}
&\dot{g}_0(t^*) = \frac{1}{m}\sqrt{1+(\tan(\phi))^2}e^{-\eta t^*}
\cos(\omega_dt^*-\phi)\\
&=\frac{1}{m}e^{-\frac{\eta}{\omega_d}(\beta +\frac{\pi}{2})}\frac{\cos(\frac{\pi}{2}+(\beta-\phi))}{\cos(\phi)}\\
&=-\frac{1}{m}e^{-\frac{\eta}{\omega_d}(\beta +\frac{\pi}{2})}\frac{\sin(\beta-\phi)}{\cos(\phi)}\\
&=-\frac{1}{m}e^{-\frac{\eta}{\omega_d}(\beta +\frac{\pi}{2})}\frac{\sin(\beta)\cos(\phi)-\cos(\beta)\sin(\phi)}{\cos(\phi)}\\
&=-\frac{1}{m}e^{-\frac{\eta}{\omega_d}(\beta +\frac{\pi}{2})}
\left(\sin(\beta)-\cos(\beta)\tan(\phi)\right)\\
&=-\frac{1}{m}e^{-\frac{\eta}{\omega_d}(\beta +\frac{\pi}{2})}
\left(
\frac{\left(\frac{1}{\tau}-\eta+\frac{r^{-1}}{d\tau}\right)}
{\Delta}
-
\frac{\omega_d}
{\Delta}
\frac{\frac{1}{2}\left(\frac{ 1}{\tau}-\frac{d}{m}\right)}{\omega_d}
\right)\\
&=-\frac{1}{m}e^{-\frac{\eta}{\omega_d}(\beta +\frac{\pi}{2})}
\left(
\frac{\left(\frac{1}{2}\left(\frac{ 1}{\tau}-\frac{d}{m}\right) +\frac{r^{-1}}{d\tau}\right)}
{\Delta}
-
\frac{\frac{1}{2}\left(\frac{ 1}{\tau}-\frac{d}{m}\right) }
{\Delta}
\right)\\
&=-\frac{1}{m}e^{-\frac{\eta}{\omega_d}(\beta +\frac{\pi}{2})}
\left(
\frac{\frac{r^{-1}}{d\tau}}
{\Delta}
\right)
\end{align*}
Finally by simplifying
\begin{align*}
\Delta&=\sqrt{\omega_d^2 + \left(\frac{1}{\tau}\!-\!\eta+\frac{r^{-1}}{d\tau}\right)^2}\\
&=
{\sqrt{   {{\frac{d+r^{-1}}{m\tau}\!-\!\frac{1}{4}\left(\frac{1}{\tau} +\frac{d}{m}\right)^2}}   + \left(\frac{1}{2}\left(\frac{1}{\tau}\!-\!\frac{d}{m}\right)+\frac{r^{-1}}{d\tau}\right)^2}}\\
&=
{\sqrt{ \frac{r^{-1}}{m\tau}  + 2\frac{1}{2}\left(\frac{1}{\tau}\!-\!\frac{d}{m}\right)\frac{r^{-1}}{d\tau}+     \left(\frac{r^{-1}}{d\tau}\right)^2}}\\
&={\sqrt{ \frac{r^{-1}}{d\tau}\frac{1}{\tau}+     \left(\frac{r^{-1}}{d\tau}\right)^2}}
\end{align*}
we get
\begin{align*}
&\dot{g}_0(t^*) =-\frac{1}{m}e^{-\frac{\eta}{\omega_d}(\beta +\frac{\pi}{2})}
\left(
\frac{\frac{r^{-1}}{d\tau}}
{\sqrt{ \frac{r^{-1}}{d\tau}\frac{1}{\tau}+     \left(\frac{r^{-1}}{d\tau}\right)^2}}
\right)>-\frac{1}{m}
\end{align*}
Therefore, since for any additional instant of time that $\dot{g}_0(t)$ achieves a local extremum $t_k^*=\frac{\beta + \frac{\pi}{2} +k\pi}{\omega_d}$, the factor $|\cos(\omega_dt_k^*-\phi)|$ does not change, then the maximum is achieved at $t=0$.

\end{proof}

\bibliographystyle{IEEETran}
\bibliography{refs}

\begin{thebibliography}{10}
\providecommand{\url}[1]{#1}
\csname url@rmstyle\endcsname
\providecommand{\newblock}{\relax}
\providecommand{\bibinfo}[2]{#2}
\providecommand\BIBentrySTDinterwordspacing{\spaceskip=0pt\relax}
\providecommand\BIBentryALTinterwordstretchfactor{4}
\providecommand\BIBentryALTinterwordspacing{\spaceskip=\fontdimen2\font plus
\BIBentryALTinterwordstretchfactor\fontdimen3\font minus
  \fontdimen4\font\relax}
\providecommand\BIBforeignlanguage[2]{{%
\expandafter\ifx\csname l@#1\endcsname\relax
\typeout{** WARNING: IEEEtran.bst: No hyphenation pattern has been}%
\typeout{** loaded for the language `#1'. Using the pattern for}%
\typeout{** the default language instead.}%
\else
\language=\csname l@#1\endcsname
\fi
#2}}

\bibitem{Evans:1924hr}
R.~D. Evans and R.~C. Bergvall, ``{Experimental analysis of stability and power
  limitations},'' \emph{Journal of the A.I.E.E.}, vol.~43, no.~4, pp. 329--340,
  Apr. 1924.

\bibitem{Steinmetz:1920fg}
C.~P. Steinmetz, ``Power control and stability of electric generating
  stations,'' \emph{Journal of the A.I.E.E.}, vol.~39, no.~2, pp. 1215--1287,
  July 1920.

\bibitem{NERC:2015tc}
{North American Electric Reliability Corporation}, ``{2015 Frequency Response
  Annual Analysis},'' Tech. Rep., Sept. 2015.

\bibitem{FederalEnergyRegulatoryCommission:2012wq}
{Federal Energy Regulatory Commission} and {North American Electric Reliability
  Corporation}, ``{Arizona-Southern California Outages on September 8, 2011},''
  Tech. Rep., Apr. 2012.

\bibitem{Hsu:di}
Y.-Y. Hsu, S.-W. Shyue, and C.-C. Su, ``Low frequency oscillations in
  longitudinal power systems: {E}xperience with dynamic stability of {Taiwan}
  power system,'' \emph{IEEE Trans. on Power Systems}, vol.~2, no.~1, pp.
  92--98, Feb. 1987.

\bibitem{klein_fundamental_1991}
M.~Klein, G.~J. Rogers, and P.~Kundur, ``{A fundamental study of inter-area
  oscillations in power systems},'' \emph{IEEE Trans. on Power Systems},
  vol.~6, no.~3, pp. 914--921, 1991.

\bibitem{messina_inter-area_2011}
M.~M. Begovic, ``Inter-area oscillations in power systems: {A} nonlinear and
  nonstationary perspective (messina, a.r.) [book reviews],'' \emph{IEEE Power
  and Energy Magazine}, vol.~9, no.~2, pp. 76--77, 2011.

\bibitem{Miller:2011tm}
N.~W. Miller, M.~Shao, and S.~Venkataraman, ``{California ISO (CAISO) Frequency
  Response Study },'' General Electric, Tech. Rep., Nov. 2011.

\bibitem{Winkelman:1981}
J.~Winkelman, J.~Chow, B.~Bowler, B.~Avramovic, and P.~Kokotovic, ``An analysis
  of interarea dynamics of multi-machine systems,'' \emph{IEEE Transactions on
  Power Apparatus and Systems}, vol. PAS-100, no.~2, pp. 754--763, 1981.

\bibitem{Verghese:1982}
G.~Verghese, I.~P\'erez-Arriaga, and F.~Schewppe, ``Selective modal analysis
  with applications to electric power systems, part ii: the dynamic stability
  problem,'' \emph{IEEE Transactions on Power Apparatus and Systems}, vol.
  PAS-101, no.~9, pp. 3126--3134, 1982.

\bibitem{bassam}
E.~Tegling, B.~Bamieh, and D.~Gayme, ``The price of synchrony: Evaluating the
  resistive losses in synchronizing power networks,'' \emph{IEEE Transactions
  on Control of Network Systems}, vol.~2, no.~3, pp. 254--266, 2015.

\bibitem{mevsanovic2016comparison}
A.~Me{\v{s}}anovi{\'c}, U.~M{\"u}nz, and C.~Heyde, ``Comparison of
  $\mathcal{H}_\infty$, $\mathcal{H}_2$, and pole optimization for power system
  oscillation damping with remote renewable generation,''
  \emph{IFAC-PapersOnLine}, vol.~49, no.~27, pp. 103--108, 2016.

\bibitem{m2016cdc}
E.~Mallada, ``{iDroop: A dynamic droop controller to decouple power grid's
  steady-state and dynamic performance},'' in \emph{55th IEEE Conference on
  Decision and Control (CDC)}, 12 2016, pp. 4957--4964.

\bibitem{poolla_dorfler2017}
B.~K. Poolla, S.~Bolognani, and F.~Dorfler, ``Optimal placement of virtual
  inertia in power grids,'' \emph{IEEE Transactions on Automatic Control},
  2017.

\bibitem{simpson-porco2017}
M.~{Pirani}, J.~W. {Simpson-Porco}, and B.~{Fidan}, ``{System-Theoretic
  Performance Metrics for Low-Inertia Stability of Power Networks},''
  \emph{ArXiv e-prints}, Mar. 2017.

\bibitem{andreasson2017}
M.~{Andreasson}, E.~{Tegling}, H.~{Sandberg}, and K.~H. {Johansson},
  ``{Coherence in Synchronizing Power Networks with Distributed Integral
  Control},'' \emph{ArXiv e-prints}, Mar. 2017.

\bibitem{jpm2017cdc}
Y.~Jiang, R.~Pates, and E.~Mallada, ``{Performance tradeoffs of dynamically
  controlled grid-connected inverters in low inertia power systems},'' in
  \emph{56th IEEE Conference on Decision and Control (CDC)}, 2017.

\bibitem{kundur_power_1994}
P.~Kundur, \emph{{Power System Stability and Control}}.\hskip 1em plus 0.5em
  minus 0.4em\relax McGraw-Hill Professional, 1994.

\bibitem{bergen_vittal_2000}
A.~Bergen and V.~Vittal, \emph{{Power System Analsysis}}.\hskip 1em plus 0.5em
  minus 0.4em\relax Prentice-Hall, 2000.

\bibitem{oakridge2013}
\BIBentryALTinterwordspacing
G.~Kou, S.~W. Hadley, P.~Markham, and Y.~Liu, ``{Developing Generic Dynamic
  Models for the 2030 Eastern Interconnection Grid},'' Oak Ridge National
  Laboratory, Tech. Rep., Dec 2013. [Online]. Available:
  \url{http://www.osti.gov/scitech/}
\BIBentrySTDinterwordspacing

\bibitem{guggilam2017engineering}
S.~S. Guggilam, C.~Zhao, E.~Dall'Anese, Y.~C. Chen, and S.~V. Dhople,
  ``Engineering inertial and primary-frequency response for distributed energy
  resources,'' \emph{arXiv preprint arXiv:1706.03612}, 2017.

\bibitem{apostolopoulou2016balancing}
D.~Apostolopoulou, P.~W. Sauer, and A.~D. Dom{\'\i}nguez-Garc{\'\i}a,
  ``Balancing authority area model and its application to the design of
  adaptive agc systems,'' \emph{IEEE Transactions on Power Systems}, vol.~31,
  no.~5, pp. 3756--3764, 2016.

\bibitem{ZhouDG}
K.~Zhou, J.~Doyle, and K.~Glover, \emph{Robust and Optimal Control}.\hskip 1em
  plus 0.5em minus 0.4em\relax Prentice Hall, 1996.

\end{thebibliography}

\end{document}